\documentclass[letterpaper, 10 pt, conference]{ieeeconf}  

\IEEEoverridecommandlockouts                              
\usepackage{graphicx} 
\usepackage{amssymb}  
\usepackage{array}
\usepackage{multirow}

\newtheorem{thm}{Theorem}
\newtheorem{proposition}[thm]{Proposition}

\newtheorem{remark}{Remark}
\newtheorem{definition}{Definition}
\newtheorem{assumption}{Assumption}
\newtheorem{question}{Question}

\title{On some open questions in bilinear quantum control
}
\author{Ugo Boscain$^{1}$, Thomas Chambrion$^2$, and Mario Sigalotti$^3$
\thanks{$^{*}$ This research has been supported by the European Research Council, ERC
StG 2009 ``GeCoMethods'', contract number 239748, by the ANR project
GCM, program ``Blanche'',
project number NT09-504490}%
\thanks{$^1$ Ugo Boscain is with 
Centre National de Recherche Scientifique (CNRS), CMAP, \'Ecole Polytechnique, Route de Saclay, 91128 Palaiseau Cedex, France, and Team GECO, INRIA-Centre de Recherche Saclay
{\tt \small ugo.boscain@polytecnique.edu}}%
\thanks{$^2$ Thomas Chambrion is with
Universit\'e de Lorraine, Institut \'Elie Cartan de Nancy, UMR 7502, Vand{\oe}uvre-l\`es-Nancy, F-54506, France, and
Inria, Villers-l\`es-Nancy, F-54600, France
{\tt \small Thomas.Chambrion@univ-lorraine.fr}}%
\thanks{$^3$ Mario Sigalotti is with
INRIA-Centre de Recherche Saclay, Team GECO and 
CMAP, \'Ecole Polytechnique, Route de Saclay, 91128 Palaiseau Cedex, France
{\tt \small mario.sigalotti@inria.fr}}%
}

\begin{document}

\maketitle
\thispagestyle{empty}
\pagestyle{empty}

\begin{abstract}
The aim of this paper is to provide a short
introduction to modern issues in the control
of infinite dimensional
closed quantum systems, driven by the bilinear Schr\"{o}dinger equation.

The first part is a quick presentation of some of the numerous
recent developments in the fields. This short summary  is intended to demonstrate
the variety of tools and approaches used by various teams in
the last decade. In a second part, we present four examples
of bilinear closed quantum systems. These examples were
extensively studied and may be used as a convenient and
efficient test bench for new conjectures. Finally, we list some open
questions, both of theoretical and practical interest.
\end{abstract}

\section{INTRODUCTION}

\subsection{Control of quantum systems}\label{SEC_physical_context}

The state of a quantum system (e.g. a charged particle) evolving on a Riemannian manifold 
$\Omega$ is described by its wave function 
$\psi$, an element of $L^2(\Omega,\mathbf{C})$. When the system is submitted to an external 
field (e.g. an electric field), 
the time evolution of the wave function is given by the Schr\"{o}dinger equation
$$
\mathrm{i}\frac{\partial \psi}{\partial t}=(-\Delta +V(x))\psi +u W(x) \psi(t), \quad x\in \Omega
$$
where $\Delta$ is the Laplace-Beltrami operator on $\Omega$, $V$ is a potential describing the 
system in absence of control, $u$ is the scalar (time variable) intensity of the external 
field and $W:\Omega\to \mathbf{R}$ is a potential accounting for the properties of the external field.

A natural question, with many practical applications, is to determine how to build 
(if it is possible) a control $u$ that steers the wave function $\psi$ from a given source to a given target.
 
\subsection{Framework and  notations}


We set the problem in a  more abstract framework. In a separable Hilbert space $H$, 
endowed with the Hermitian product $\langle,\rangle$, we consider the following control system 
\begin{equation}\label{EQ_main}
\frac{d}{dt} \psi=(A+u(t) B) \psi,
\end{equation}
where $(A,B)$ satisfies Assumption \ref{ASS_1}.

\begin{assumption}\label{ASS_1}
$(A,B)$ is a pair of (possibly unbounded) linear operators in $H$ such that
\begin{enumerate}
\item $A$ is skew-adjoint on its domain $D(A)$;
\item there exists a Hilbert basis $(\phi_k)_{k\in \mathbf{N}}$ of $H$ made of eigenvectors 
of $A$: for every $k$, $A \phi_k =\mathrm{i} \lambda_k \phi_k$ with $\lambda_k$ in $\mathbf{R}$ and $\lambda_k$ tends to $-\infty$ as $k$ tends to $\infty$ ;
\item for every $j$ in $\mathbf{N}$, $\phi_j$ belongs to  $D(B)$, the domain of $B$;
\item there exists $U \subset \mathbf{R}$ containing at least $0$ and $1$ such that $A+u B $ 
is essentially skew-adjoint (not necessarily with domain $D(A)$) for every $u$ in $U$;
\item $\langle B \phi_j,\phi_k\rangle=0$ for every $j,k$ in $\mathbf{N}$ such that 
$\lambda_j=\lambda_k$ and $j\neq k$. 
\end{enumerate} 
\end{assumption}

If $(A,B)$ satisfies Assumption \ref{ASS_1}, for every $u$ in $U$, $A+uB$ generates a unitary 
group of propagators $t\mapsto e^{t(A+uB)}$. By concatenation, one can define the solution of 
(\ref{EQ_main}) for every piecewise constant functions $u$ taking value in $U$, for every 
initial condition $\psi_0$ given at time $t_0$. We denote this solution 
$t\mapsto \Upsilon^u_{t,t_0} \psi_0$. To the best of our knowledge,  it is not possible to 
define the propagator $\Upsilon^u$ for controls $u$ that are not piecewise constant in 
the general framework of Assumption \ref{ASS_1}. With some extra regularity assumptions, 
it is  possible to extend the definition of $\Upsilon^u$ to more general controls. For instance, 
if $B$ is bounded, $\Upsilon$ admits a continuous extension to the set $L^1(\mathbf{R},\mathbf{R})$
(see \cite[Proposition 1.1]{MR2491595}).

The framework of Assumption \ref{ASS_1} is, in one sense, too  general for the purpose of 
quantum mechanics. For instance, it includes the example of Section \ref{SEC_physical_context} 
with $H=L^2(\Omega,\mathbf{C})$ and $V$ any $L^{\infty}$ function.  Following Cohen-Tannoudji et 
al., \cite[Figure 7a, page 35 and Section II-A-1, page 94]{Cohen}, one of the most physically 
relevant cases is precisely the one where the potentials $V$ and $W$ and the wave functions are smooth. 
\begin{quote}
 ``From a physical point of view, it is clear that the set $L^2(\Omega,\mathbf{C})$ is too wide in scope:
given the meaning attributed to $|\psi(x, t)|^2$, the wave functions which are actually
used possess certain properties of regularity. We can only retain the functions $\psi(x, t)$ 
which are everywhere defined, continuous, and infinitely differentiable''
\end{quote}
This is the main motivation for the notion of weak-coupling (see \cite{weakly-coupled}).
\begin{definition}\label{DEF_weak_coupling}
Let $k>0$.
A pair $(A,B)$ satisfying Assumption \ref{ASS_1} is $k$-weakly-coupled  if
\begin{enumerate}
\item for every $u$ in $\mathbf{R}$, $A+uB$ is skew-adjoint with domain $D(A)$;
\item  for every $u\in {\mathbf R}$, $D(|A+ u B|^{k/2})=D(|A|^{k/2})$;
\item there exists $d\geq 0$ and $r<k$ such that $\|B\psi\|\leq d\||A|^{r/2}\psi\|$ for every $\psi$ in $D(|A|^{k/2})$;   
\item  there exists
a constant $C$ such that,  for every $\psi$ in $D(|A|^k)$, $ |\Re \langle |A|^k
\psi,B\psi \rangle |\leq C |\langle |A|^k \psi, \psi \rangle|$.
\end{enumerate}
\end{definition}

If $(A,B)$ is $k$-weakly-coupled, the \emph{coupling constant}  $c_k(A,B)$ of system $(A,B)$ of
order $k$ is the quantity
$$
\sup_{\psi\in D(|A|^{k})|A\psi \neq 0}  \frac{ |\Re
\langle |A|^k
\psi,B\psi \rangle |}{|\langle |A|^k \psi, \psi \rangle|}.
$$

We denote by $PC(U)$ the set of piecewise constant functions $u$ such that there exists two 
sequences $0=t_1 <t_2 <\ldots < t_{p+1}$ and $u_1,u_2,\ldots,u_p$ in $U\setminus\{0\}$ with
$$
u=\sum_{j=1}^p u_j \mathbf{1}_{[t_j,t_{j+1})}.
$$

The operators $A$ and $B$  can be represented by infinite dimensional matrices in the 
basis $(\phi_k)_{k \in \mathbf{N}}$. For every $j,k$,  we denote 
$b_{jk}=\langle \phi_j, B \phi_k \rangle$. For every $N$, the orthogonal projection 
$\pi_N:H\rightarrow H$ on the space spanned by the first  $N$ eigenvectors of $A$ is defined by
$$
\pi_N(x)=\sum_{k=1}^N \langle \phi_k,x\rangle \phi_k \quad \quad \mbox{for every } x \mbox{ in } H.
$$
Let $\mathcal{L}_N$ be the range of $\pi_N$. 
The \emph{compressions} of $A$ and $B$ at order $N$ are the finite rank operators 
$A^{(N)}=\pi_N A_{| \mathcal{L}_N}$ and $B^{(N)}=\pi_N B_{| \mathcal{L}_N}$  respectively. 
The \emph{Galerkin approximation}  of (\ref{EQ_main})
of order $N$ is the system in $\mathcal{L}_N$
\begin{equation}\label{eq:sigma}
\dot x = (A^{(N)} + u B^{(N)}) x,
\end{equation}
whose propagator is denoted with $X^u_{(N)}$.

A pair $(j,k)$ in $\mathbf{N}^2$ is a \emph{non-degenerate} (also called \emph{non-resonant}) transition of $(A,B)$ if $b_{jk}\neq 0$  and,  for every $l,m$,  
$|\lambda_j-\lambda_k|=|\lambda_l-\lambda_m|$ implies $\{j,k\}=\{l,m\}$ or $\{l,m\}\cap\{j,k\}=\emptyset$.

A subset $S$ of $\mathbf{N}^2$ is a \emph{chain of connectedness} of $(A,B)$ if  for 
every $j,k$ in $\mathbf{N}$, there exists a finite sequence $p_1=j,p_2,\ldots,p_r=k$ for which 
$(p_l,p_{l+1})\in S$ and $\langle \phi_{p_{l+1}}, B \phi_{p_l} \rangle \neq 0$ for every $l=1,\ldots,r-1$. A chain 
of connectedness $S$ of $(A,B)$ is \emph{non-degenerate} if every element of $S$ is a 
non-degenerate tran\-sition of $(A,B)$.

\subsection{Content of the paper}
Sections \ref{SEC_exact_control} and \ref{SEC_approx_control} present a short review of results a\-vaila\-ble in 
the literature about exact and approximate controllability of infinite dimensional bilinear quantum systems.  
Section \ref{SEC_example} collects four examples of bilinear quantum systems that were extensively studied in the last decade. Finally, we suggest 
five questions in Section  \ref{SEC_questions} 
that we think both important and natural.

\section{EXACT CONTROLLABILITY}\label{SEC_exact_control}

\subsection{Obstructions to exact controllability}\label{SEC_exact_negative}

The first result about bilinear control is a general negative result due to Ball, Marsden 
and Slemrod \cite{BMS}. It was adapted to the case of bilinear quantum systems by Turinici 
\cite{turinici} in the following form:
\begin{proposition}[\cite{turinici}]\label{PRO_BMS}
Let $(A,B)$ satisfy Assumption \ref{ASS_1} and $B$ be bounded. Then, for every $r>1$, 
for every $\psi_0$ in $D(A)$, the attainable set from $\psi_0$ with controls in $L^r$, 
$\{\Upsilon^u_{t,0}\psi_0| u\in L^{r}(\mathbf{R},\mathbf{R})\}$  is a countable union 
of closed sets with  empty interior  in $D(A)$. In particular, this attainable set has  
empty interior  in $D(A)$.
\end{proposition}
Proposition \ref{PRO_BMS} admits a natural extension in the case of weakly-coupled systems:
\begin{proposition}[\mbox{\cite[Proposition 2]{weakly-coupled}}]\label{PRO_obstruction_weak_coupling}
Let $(A,B)$ be $k$ weakly-coupled and $B$ be bounded.  Then, for every $\psi_0$ in $D(|A|^{k/2})$, 
for every $u$ in $L^1(\mathbf{R}, \mathbf{R})$, for every $t\geq 0$, $\Upsilon^u_{t,0}\psi_0$ 
belongs to $D(|A|^{k/2})$. In particular,   $\{\Upsilon^u_{t,0}\psi_0| u\in L^{1}(\mathbf{R},\mathbf{R})\}$ 
has empty interior in $D(|A|^{r/2})$ for every $r<k$.
\end{proposition}
Most of the bilinear quantum systems encountered in the literature are $k$ weakly-coupled for 
every $k>0$. Notice that the eigenvectors of $A$ are in $D(|A|^k)$ for every $k$. As a 
consequence, the attainable set for such a system from any eigenvector of $A$ is contained 
in $\cap_{k>0} D(|A|^k)$, the intersection of all the iterated domains of $A$. 

\subsection{Attainable set of the infinite square potential well}\label{b-c}

The results of Section \ref{SEC_exact_negative} do not exclude exact controllability on a 
sufficiently small subset of $H$. In a series of paper (\cite{beauchard,camillo}), Beauchard 
\emph{et al.} determined the attainable set for the infinite square potential well.
\begin{thm}\label{THM_beauchard}
Consider the bilinear Schr\"{o}dinger equation
$$ \mathrm{i}\frac{\partial \psi}{\partial t}=-\Delta\psi +u x \psi(t), \quad x\in  (0,1).$$
The attainable set with $L^2$ controls from the first eigenstate of the Laplacian is 
exactly the intersection of the unit sphere of $L^2((0,1),\mathbf{C})$ with 
$H^3_{(0)}=\{\psi \in H^3((0,1),\mathbf{C})\mid \psi(0)=\psi(1)=\psi''(0)=\psi''(1)=0\}$.
 \end{thm}

\section{APPROXIMATE CONTROLLABILITY}\label{SEC_approx_control}
\subsection{Lyapunov techniques}

Because of the specific features of quantum systems and, in particular, of the effects 
of the measurements on its evolution, classical closed-loop control strategies cannot be 
directly implemented in the Schr\"odinger framework. Nevertheless, the strategy consisting 
in identifying a Lyapunov function that measures the distance from the desired final state 
(or the distance from a trajectory that one wants to track) and that can be forced to 
decrease towards zero by a suitable state-dependent choice of the control parameter can be 
used to obtain, via simulation, open-loop control laws that approximately steer the system 
towards the prescribed goal. 
This approach has been explored in \cite{MR2168664} and refined in \cite{Mirrahimi}
for systems evolving in a finite-dimensional Hilbert space $H$. The proof of the convergence towards the goal 
adapts the classical Jurdjevic--Quinn method \cite{JurdjevicQuinn} and  
is based on the  LaSalle invariance principle.

In the case where $H$ is infinite-dimensional, generalizations of the previously mentioned results 
have been obtained by suitably adapting   LaSalle invariance principle
(see, in particular, 
\cite{Nersesyan}, \cite{MR2491595}, \cite{ito-kunisch}, \cite{beauchard-nersesyan}
for  the case where the drift operator of the 
 bilinear Schr\"odinger equation has discrete-spectrum).

\subsection{Geometric techniques: general case}

\begin{definition}
Let $(A,B)$ satisfy Assumption \ref{ASS_1}, $(j,k)$ be a pair of integers such that 
$\lambda_j\neq \lambda_k$ and $u^\ast:\mathbf{R}\to U$ be $T=2\pi/|\lambda_j-\lambda_k|$-periodic 
and not almost everywhere zero. The number 
$$\mathrm{Eff}_{(j,k)}(u^\ast)=
\frac{\left | \int_0^T\!\!u^{\ast}(\tau)e^{\mathrm{i}(\lambda_j-\lambda_k)\tau}\mathrm{d}\tau
 \right |}{ \int_0^T\!\! |u^{\ast}(\tau)|\mathrm{d}\tau} $$ 
is called the \emph{efficiency} of $u^\ast$ with respect to the transition $(j,k)$ of $(A,B)$. 
\end{definition}

\begin{proposition}[\mbox{\cite[Theorem 1]
{periodic}}]\label{PRO_periodic}
Let $(A,B)$ satisfy Assumption \ref{ASS_1} and $U$ be such that $ U/n \subset U$ for every $n\in \mathbf{N}$. 
Let $(1,2)$ 
be a non-degenerate transition of $(A,B)$. If  $u^\ast:\mathbf{R}\to U$ is 
$2\pi/|\lambda_1-\lambda_2|$-periodic with $\mathrm{Eff}_{(1,2)}(u^\ast)\neq 0$ and 
$\mathrm{Eff}_{(j,k)}(u^\ast) = 0$ for every $j,k$ such that $|\lambda_j-\lambda_k|\in \mathbf{N}|\lambda_2-\lambda_1|$  
and $\{j,k\}\neq \{1,2\}$, then there exists $T^\ast>0$ such that 
$|\langle \phi_2,\Upsilon^{u^\ast/n}_{nT^\ast,0}\phi_1 \rangle |$ tends to 1 as $n$ tends to infinity.
\end{proposition}

\begin{proposition}\label{pro_approx_contr_norme_H}
Let $(A,B)$ satisfy Assumption \ref{ASS_1} and admit a non-degenerate chain of connectedness. 
Then, for every $\varepsilon>0$, for every unitary operator $\hat{\Upsilon}$ in $\mathbf{U}(H)$, 
for every $n$ in $\mathbf{N}$, there exists a piecewise constant function 
$u_\varepsilon:[0,T_\varepsilon]\to U$ such that $\|\Upsilon^u_{T_\varepsilon,0}\phi_j-\hat{\Upsilon}\phi_j\|<\varepsilon$ 
for $1\leq j \leq n$.  
\end{proposition}
\begin{proof}
The original proof given in \cite{Schrod2} is a particular case of Proposition \ref{PRO_periodic} 
(see \cite[Proof of Lemma 4.3]{Schrod2} for an explicit construction of $u^\ast$). This proof is valid if $U$ accumulates 
at zero.  Thanks to \cite[Proposition 3]{quadratic}, one can replace the sequence $u^\ast/n$ by a 
sequence of controls taking value in $\{0,1\}$. 
\end{proof}

\begin{proposition}
Let $(A,B)$ satisfy Assumption \ref{ASS_1} and admit a non-degenerate chain of connectedness $S$. 
Then, for every $\varepsilon>0$, for every $(j,k)$ in $S$, there exists a piecewise constant 
function $u_\varepsilon:[0,T_\varepsilon]\to U$ such that 
$\|\Upsilon^u_{T_\varepsilon,0}\phi_j-\phi_k\|<\varepsilon$  and 
$$\|u_\varepsilon\|_{L^1} \leq  \frac{5\pi}{4|\langle \phi_k,B\phi_j\rangle|}.$$
\end{proposition}
\begin{proof}
In the case where $U$ accumulates at zero, this is \cite[Proposition 2.8]{Schrod2}. The general case 
$U=\{0,1\}$ follows from \cite[Proposition 3]{quadratic}.
\end{proof}
A lower bound of the $L^1$ norm of the control needed to induce  a transfer from a wave 
function $\psi_a$ in the unit sphere of $H$ is given by \cite[Proposition 4.6]{Schrod}:
$$
\sup_{n\in \mathbf{N}} \frac{\big | |\langle \phi_n, \psi_a \rangle | -  
|\langle \phi_n, \Upsilon^u_{t,0} \psi_a \rangle | \big |}{\|B \phi_n\|} \leq \int_0^t \!\!\! |u(\tau)|\mathrm{d}\tau
$$
for every $(A,B)$ satisfying Assumption \ref{ASS_1}, 
every $t\geq 0$,  and every piecewise constant $u$ taking value in $U$.

\subsection{Geometric techniques: weakly-coupled systems}

\begin{proposition}[\mbox{\cite[Proposition 2]{weakly-coupled}}]\label{PRO_croissance_norme_A} Let  $(A,B)$ satisfy Assumption \ref{ASS_1} 
and be $k$-weakly-coupled.  Then,
for every $\psi_{0} \in D(|A|^{k/2})$, $K>0$,
$T\geq 0$, and $u$ piecewise constant such that
$\|u\|_{L^1}< K$, one has
$\left\||A|^{\frac{k}{2}}\Upsilon^{u}_{T}(\psi_{0})\right\| \leq
e^{c_k(A,B) K} \| |A|^{\frac{k}{2}}\psi_0 \|.$
\end{proposition}
\begin{proposition}[\mbox{\cite[Proposition 4]{weakly-coupled}}]\label{prop:gga}
 Let $k$ in $\mathbf{N}$ and $(A,B)$  satisfy Assumption \ref{ASS_1} and be $k$-weakly-coupled.
Then
for every $\varepsilon > 0 $, $s<k$, $K> 0$, $n\in \mathbf{N}$, and
$(\psi_j)_{1\leq j \leq n}$ in $D(|A|^{k/2})^n$
there exists $N \in \mathbf{N}$
such that
for every piecewise constant function $u$
$$
\|u\|_{L^{1}} < K \Rightarrow\| |A|^{\frac{s}{2}}(\Upsilon^{u}_{t}(\psi_{j}) -
X^{u}_{(N)}(t,0)\pi_{N} \psi_{j})\|\! <\! \varepsilon,
$$
for every $t \geq 0$ and $j=1,\ldots,n$.
\end{proposition}
\begin{remark}
Notice that, in Propositions~\ref{PRO_croissance_norme_A} and~\ref{prop:gga}, the upper bound of the 
$|A|^{k/2}$ norm of the solution of (\ref{EQ_main}) or the bound on the error between the infinite dimensional 
system and its finite dimensional approximation only depend on the $L^1$ norm of the control, not on the time. 
\end{remark}
The a priori bound for the $|A|^{k/2}$ norm combined with an interpolation argument allows 
to deduce approximate controllability in $|A|^{r/2}$ norm from the approximate controllability 
in $A^0$ norm (i.e., the norm of $H$):
\begin{proposition}
Let  $(A,B)$ satisfy Assumption \ref{ASS_1}, be $k$-weakly-coupled and admit a non-degenerate chain of connectedness. 
Then, for every $\varepsilon>0$, for every $n$ in $\mathbf{N}$, for every unitary operator $\hat{\Upsilon}$ in $\mathbf{U}(H)$, for every $r<k/2$,
there exists $u_\varepsilon:[0,T_\varepsilon]\to \{0,1\}$ such that
$\||A|^r(\hat{\Upsilon}\phi_j -\Upsilon^{u_\varepsilon}_{T_\varepsilon,0}\phi_j )\| <\varepsilon$ for $j\leq n$. 
\end{proposition}
\begin{proof}
This would be \cite[Proposition 5]{weakly-coupled} if the controls $u_\varepsilon$ took value in $(0,1)$. 
The 
case $U=\{0,1\}$ follows from \cite[Proposition 3]{quadratic}.
\end{proof}

\subsection{Other results}

Let us mention in this section 
some other 
results concerning 
quantum control problems on infinite-dimensional spaces 
which do not
satisfy Assumption~\ref{ASS_1}. 

First of all, some papers deal with the case where the drift Hamiltonian has some 
continuous spectrum and consider the problem of approximately controlling between 
the eigenstates corresponding to the discrete part of the spectrum. In particular, 
in \cite{mirrahimi-continuous} Mirrahimi considers the case of a drift operator of 
the form $-\Delta+V$ on $\mathbf{R}^d$, where $V$ is a potential decaying at infinity. 
The controllability is proved using a Lyapunov technique and estimating the interaction 
with continuum spectrum thanks to Strichartz estimates. 

Another important class of systems exhibiting continuous spectrum is obtained by considering 
the ensemble control of Bloch equations. 
The corresponding system consists in a continuum of finite-dimensional systems 
coupled by the control parameter only. Each system of the ensemble is parameterized 
by a characteristic frequency.
Controllability results in this setting have been obtained  in \cite{LiKhanejaPRA,LiKhanejaTAC,BeauchardCoronRouchon}.

Other interesting class of problems is given by models for 
a quantum oscillator coupled with a spin  
(see \cite{rangan,puel}).
The spectrum in this case is discrete, but it \emph{intrinsically} presents degenerate transitions. 
The controllability results are obtained exploiting the 
presence of more than one control. 

Let us finally mention the widely used adiabatic methods. They require the use of several controls 
(not only one, as in Equation (\ref{EQ_main})) and rely on adiabatic theory and the interesections of eigenvalues. 
Approximate controllability is obtained through slow variations of the different controls (see \cite{adiabatic}).

\section{FOUR EXAMPLES}
\label{SEC_example}

\subsection{Infinite square potential well}
The first example we consider describes a particle confined in a 1D box $(0,\pi)$. 
 This model has been extensively studied
by several authors in the last few years and it has been  the first quantum system for 
which a positive controllability
 result has been obtained. Beauchard proved exact controllability in some dense subsets 
of $L^2$ first using Coron's return method (\cite{beauchard}), then standard linear 
test (\cite{camillo}). Nersesyan obtained approximate controllability results using 
Lyapunov techniques (\cite{Nersesyan,beauchard-nersesyan}), which allowed to obtain the global result (i.e., Theorem \ref{THM_beauchard} recalled in Section~\ref{b-c}).

The Schr\"{o}dinger equation writes
\begin{equation}\label{EQ_potential_well}
 \mathrm{i} \frac{\partial \psi}{\partial t}=-\frac{1}{2} \frac{\partial^2 \psi}{\partial x^2} - u(t) x \psi(x,t)
\end{equation}
with boundary conditions $\psi(0,t)=\psi(\pi,t)=0$ for every $t \in \mathbf{R}$.

With our notations, ${H}=L^2 \left ( (0,\pi), \mathbf{C} \right )$ endowed with 
the Hermitian product $\langle \psi_1,\psi_2\rangle= \int_{0}^{\pi} \overline{\psi_1(x)} \psi_2(x) dx$. 
The operators $A$ and $B$ are defined by $A\psi= \mathrm{i}\frac{1}{2} \frac{\partial^2 \psi}{\partial x^2}$ 
for every $\psi$ in
$D(A)= (H_2 \cap  H_0^1 )\left((0,\pi), \mathbb{C}\right)$, and $B\psi:x\mapsto \mathrm{i} x \psi(x)$.
A Hilbert basis of $H$ is $(\phi_k)_{k\in \mathbf{N}}$ with $\phi_k:x\mapsto \sin(kx)/\sqrt{2}$. 
For every $k$, $A\phi_k=-\mathrm{i}(k^2/2) \phi_k$.

For every $j,k$ in $\mathbf{N}$, 
$$
b_{jk}=\langle \phi_j, B \phi_k\rangle =
\left\{
\begin{array}{ll}
(-1)^{j+k}\frac{ 2 jk}{(j^2-k^2)^2} & \mbox{ if } j-k \mbox{ odd}\\
0 & \mbox{otherwise}.
\end{array}
\right.
$$

Despite numerous degenerate transitions, the system is approximately 
controllable (see \cite[Section 7]{Schrod2}).

One can directly check that $(A,B)$ is $2$-weakly-coupled.
By Proposition \ref{PRO_obstruction_weak_coupling}, the system cannot be $k$-weakly-coupled 
for $k>3$ (since the attainable set from any eigenvector of $A$ contains the intersection of 
the unit $L^2$ sphere with $H^3_{(0)}=D(|A|^{3/2})$).  

For example of control designs and numerical simulations, we refer to \cite[Section IV]{QG}.
\subsection{Harmonic oscillator}
The quantum harmonic oscillator is among the most important examples of quantum
system (see, for instance, \cite[Complement $G_V$]{cohen77}). Its controlled
version has been extensively studied (see, for instance, \cite{Rouchon,illner}).
In this example $H=L^2(\mathbf{R},\mathbf{C})$
and equation (\ref{EQ_main}) reads
\begin{equation}\label{EQ_harmonic_oscillator}
\mathrm{i}\frac{\partial \psi}{\partial t}(x,t)=\frac{1}{2}(-\Delta
+x^2)\psi(x,t) +u(t) x \psi(x,t).
\end{equation}
A Hilbert basis of $H$ made of eigenvectors of $A$ is given by the sequence of
the Hermite functions $(\phi_n)_{n \in \mathbf{N}}$, associated with the
sequence $( - \mathrm{i} \lambda_n)_{n \in \mathbf{N}}$ of eigenvalues where
$\lambda_n=n-1/2$ for every $n$ in $\mathbf{N}$. In the basis $(\phi_n)_{n \in
\mathbf{N}}$, $B$ admits a tri-diagonal structure
$$
\langle \phi_j,B\phi_k\rangle = \left \{\begin{array}{cl}
- \mathrm{i} \sqrt{\frac k2} & \mbox{if } j=k-1,\\
- \mathrm{i}\sqrt{\frac{k+1}2} & \mbox{if } j=k+1,\\
0 & \mbox{otherwise.}
\end{array} \right.
$$
For every $k$ in $\mathbf{N}$, the system $(A,B)$ is
$k$-weakly-coupled (see \cite{weakly-coupled}) and
\begin{eqnarray*}
c_k(A,B) &\leq &3^{k}-1.
\end{eqnarray*}

The quantum harmonic oscillator is not controllable (in any reasonable
sense) as proved in~\cite{Rouchon}.
 However, the Galerkin approximations of
(\ref{EQ_harmonic_oscillator}) of every order are exactly controllable
(see~ \cite{PhysRevA.63.063410}), and Proposition~\ref{prop:gga} ensures that any trajectory 
of the infinite dimensional system is a uniform limit of trajectories of its Galerkin 
approximations. This is not a
contradiction, since Proposition~\ref{prop:gga} does not say that  every trajectory of every 
Galerkin approximation  is close to the 
trajectory of the infinite-dimensional system having the same initial condition and corresponding to the same control. 
What happens for the quantum oscillator is that 
if one wants
to steer 
the Galerkin approximation of order $N$ of (\ref{EQ_harmonic_oscillator}) 
from a given state (say, the first eigenstate) to an
$\varepsilon$-neighbourhood of another given target (say, the second eigenstate), the $L^1$ norm of 
the control blows up as $N$ tends to infinity.
It is compatible with Proposition~\ref{prop:gga} that the sequence of these trajectories does not converge to a trajectory of 
(\ref{EQ_harmonic_oscillator}). 


To obtain an estimate of the order $N$ of the Galerkin approximation whose
dynamics remains  $\varepsilon$ close to the one of the infinite dimensional
system when using control with $L^1$-norm $K$, one can  use
 \cite[Remark 8]{weakly-coupled} and we find that $\|X^{(N)}_u(t,0) \phi_1 -
\pi_N \Upsilon^u_t \phi_1\|\leq \varepsilon$ provided $\|u\|_{L^1} \leq K$ and
$$
\frac{2^{N-1} \sqrt{N+2}}{(N-1)!}\sqrt{\frac{(2N)!}{(N+1)!}} K^N<\varepsilon.
$$
For instance, if $K=3$ and $\varepsilon=10^{-4}$, this is true for $N= 413$.

\subsection{Planar rotation of a linear molecule}
The next example involves a bilinear Schr\"{o}dinger equation on a manifold with non-trivial topology.

We consider  a rigid bipolar molecule rotating in a plane. Its only
degree of freedom is the rotation around its centre of mass.  The molecule
is submitted to an electric field of constant direction with variable intensity
$u$.  The orientation of the molecule is an angle in $\Omega=SO(2) \simeq
\mathbf{R}/2\pi \mathbf{Z}$. The dynamics is governed by the Schr\"odinger
equation
$$
\mathrm{i}\frac{\partial \psi(\theta,t) }{\partial t} =
\left(-\frac{\partial^2}{\partial \theta^2} + u(t)\cos \theta \right)
\psi(\theta, t), \quad \theta \in \Omega.
$$
Note that the parity (if any) of the wave function is preserved by the above
equation.  We consider then the Hilbert space $H=\{\psi\in
L^2(\Omega,\mathbf{C}): \psi \mbox{ odd} \}$, endowed with the Hilbert product
$\langle f,g\rangle =\int_{\Omega} \bar{f}g$. The eigenvalue of the skew-adjoint
operator $A = \mathrm{i}\frac{\partial^2}{\partial \theta^2}$
associated with the eigenfunction $\phi_k:\theta\mapsto\sin(k\theta)/\sqrt{\pi}$
is $-\mathrm{i} \lambda_k = - \mathrm{i} k^2$, $k \in \mathbf{N}$.
The domain of $|A|^{k}$ is the Hilbert space
$H^k_e=\{\psi\in H^{2k}(\Omega,\mathbf{C}): \psi \mbox{ odd}
\}$. The skew-symmetric operator $B=-\mathrm{i} \cos \theta$ is bounded on
$D(|A|^{k/2})$ for every $k$.
For every $k$ in $\mathbf{N}$, $(A,B)$ is
$k$-weakly-coupled (\cite[proposition 8]{weakly-coupled}).
For every $k$ in $\mathbf{N}$, $c_k(A,B)\leq\frac{2^{2k}-1}{2}$.

From the point of view of the controllability problem, notice that the operator $B$ couples only
adjacent eigenstates, that is, $\langle
\phi_l, B \phi_j \rangle = 0$ if and only if $|l-j|>1$. Since $\lambda_{l+1} -
\lambda_l = 2l+1$ then $\{(j,l)\in \mathbf{N}^2 \,:\, |l-j|=1\}$ is
a non-degenerate connectedness chain for $(A,B)$. Therefore, by 
\cite[Proposition 5]{weakly-coupled} the system provides an example of approximately
controllable system in norm $H^k(\Omega,\mathbf{C})$ for every
$k$. Note that, since the eigenstates are in $H^{k}(\Omega,\mathbf{C})$ for every $k$ then
the
reachable set from any eigenstate is contained in $H^{k}(\Omega,\mathbf{C})$ for every
$k$.
%

\subsection{Everywhere dense attainable set and no Good Galerkin Approximation}\label{SEC_Anharmonic}
To the best of our knowledge, the following academic example does not appear in the 
physics literature. For $\alpha$ in $\mathbf{N}$, consider the following bilinear Schr\"{o}dinger equation 
\begin{eqnarray}
\lefteqn{\mathrm{i} \frac{\partial \psi}{\partial t}=
\left \lbrack \left ( -\frac{1}{2} \Delta +x^2 \right )^\alpha +  
\left ( -\frac{1}{2} \Delta +x^2 \right )^{-1} \right \rbrack \psi  }&&  \nonumber\\
&&\quad \quad + u(t) x^4 \psi(x,t), \qquad  x \in \mathbf{R}. \quad \quad \quad 
\end{eqnarray}

This strongly perturbed harmonic oscillator checks the controllability conditions of 
\cite[Proposition 2.8]{Schrod2}, with $H$ equal to the set of even $L^2$ functions on 
$\mathbf{R}$, $A=-\mathrm{i}[(-\Delta/2+x^2)^\alpha+(-\Delta/2+x^2)^{-1}]$ 
and $B=-\mathrm{i}x^4$ . The bilinear Schrödinger equation is well-posed for piecewise 
constant nonnegative controls $u$ (see \cite[Theorems XIII.69 and XIII.70]{reed-simon-3}). 
A basis of $H$ made of eigenvectors of $A$ is given by $(\phi_{2n})_n$ where $\phi_k$ is 
the $k^{th}$ Hermite function.   A non-degenerate chain of connectedness of $(A,B)$ is 
$\{(j,j+1),j\in \mathbf{N}\}$. Since $|\langle \phi_j,B \phi_{j+1} \rangle |\sim j^{-2}$, 
$(|\langle \phi_j,B \phi_{j+1} \rangle |^{-1})_j \in \ell^1$. As a consequence, it is possible 
to join (approximately) any energy level from the first one with a control of $L^1$ norm 
less than $(5\pi/4) \sum_{j \in \mathbf{N}}|\langle \phi_j,B \phi_{j+1} \rangle |^{-1}<+\infty$. 
Hence the system does not admit Good Galerkin Approximations in the spirit of Proposition \ref{prop:gga}, since it is possible to reach arbitrary high energy levels using controls with a given finite $L^1$ norm.

\section{FIVE OPEN QUESTIONS}\label{SEC_questions}
\addtolength{\textheight}{-2.1cm}
\subsection{Attainable set of weakly-coupled systems}
Most of the bilinear quantum systems we encountered in the physics literature are 
$k$-weakly-coupled for every $k>0$.
We have already seen that if $(A,B)$ is $k$-weakly-coupled for every $k>0$, then the 
attainable set from any eigenvector of $A$ is contained in $\cap_{k>0} D(|A|^k)$, the 
intersection of the domains of all the iterations of $A$. This prevents a direct application 
of the linear test used in \cite{camillo} because of the difficulty to endow  
$\cap_{k>0} D(|A|^k)$ (or a subspace of it)  with a Banach structure. But it does not forbid (a priori) 
 the eigenstates of $A$ to be in the attainable set of the first eigenstate. The complete description
of the attainable set from the first eigenstate is likely out of reach without new powerful methods. 
One may consider the less challenging 
\begin{question}
Let $(A,B)$ be  $k$-weakly-coupled for every $k>0$. Give (explicitly) a state $\psi_b$ not colinear 
to $\phi_1$ such that there exist a control $u$ in $L^1(\mathbf{R},\mathbf{R})$ and a 
time $T>0$ for which $\Upsilon^u_{T,0}\phi_1=\psi_b$.
\end{question}

\subsection{Minimal time}
Let $(A,B)$ satisfy Assumption \ref{ASS_1} and admit a non-degenerate chain of connectedness. 
From Proposition \ref{pro_approx_contr_norme_H}, we know that 
$\cup_{t\geq 0}\{\Upsilon^u_{t,0}\phi_1|u \in PC\}$ is dense in $H$. We define 
$$
\rho=\inf \left \{\! T\geq 0 \mbox{ such that } \overline{\cup_{0\leq t\leq T} \{\Upsilon^u_{t,0}\phi_1|u \in PC\}}=H \right \}\!.
$$
It is classical that $\rho>0$  if $A$ is bounded, which is the case for instance if 
$H$ is finite dimensional. The computation of $\rho$ is difficult in practice. At present 
time, $\rho$ is unknown for all the examples of Section \ref{SEC_example}. 
An example ($H=L^2(\mathbf{R}/2\pi\mathbf{Z},\mathbf{C})$, $A= \mathrm{i}(-\Delta)^\alpha$ 
with $\alpha>5/2$, $B: \psi \mapsto \mathrm{i} \cos(\theta) \psi $) has been 
recently exhibited for which $\rho=0$, see \cite{Time}. 
\begin{question}
Does it exist $(A,B)$ $k$-weakly-coupled for every $k>0$ such that $A$ is unbounded, 
$B$ has no eigenvector and $\rho>0$?
\end{question}
A related question has been investigated by Beauchard and Morancey in \cite{arXiv:1208.5393}, 
where they give a set of sufficient conditions for the attainable set of a $3$-weakly-coupled 
system in small time with small controls  to have empty or non-empty  interior in $D(|A|^{3/2})$.

\subsection{Transfer time and size of controls}

As previously said, large controls may, for some examples, allow approximate 
controllability in arbitrarily small time. For weakly-coupled systems, it can be easily proved 
(see \cite{Time}) that an a priori bound on the $L^1$ norm of the control is not compatible with approximate controllability 
in arbitrarily small time. In practice (in particular when using adiabatic methods), one often 
applies very small controls, what results in large transfer time.
\begin{question}
 An upper bound on the $L^1$ norm of the control being given, 
what is the smallest possible time needed to transfer a given system $(A,B)$ from the first eigenstate of $A$ to the second one?  
\end{question}

\subsection{Minimal number of switches}

 In the case where $B$ is bounded, the following  computation
\begin{eqnarray*}
\|A e^{t(A+uB)}\psi\| &=&\|(A+uB-uB) e^{t(A+uB)}\psi\| \\
	&\leq & \|(A+uB) e^{t(A+uB)}\psi\| + |u| \| B\|\\
	&\leq &  \| e^{t(A+uB)}(A+uB) \psi\| + |u| \| B\|\\
	&\leq & \|A \psi\| + 2 |u| \|B\|,
\end{eqnarray*}
valid for every $u$ in $U$, $t\geq 0$ and  $\psi$ in the intersection of the unit sphere 
of $H$ and $D(A)$, gives an upper bound of variation of the energy of the system in term 
of the total variation of the control $u$. This provides a lower bound of the number of 
discontinuities of a piecewise constant control taking value in $\{0,1\}$ to reach a given target.

Let $(A,B)$ satisfy Assumption \ref{ASS_1}. If $(A,B)$ admits a non-degenerate chain of connectedness, then for 
every $\psi_b$ in the unit sphere of $H$, for every $\varepsilon>0$, there exists 
$u_\varepsilon:[0,T_\varepsilon]\to \{0,1\}$ such that $\|\Upsilon^u_{T_u,0}\phi_1 -\psi_b \|<\varepsilon$. 
Using  \cite[Proposition 3]{quadratic}, it is possible to build $u_\varepsilon$ with a number 
of discontinuities of the order of $1/\varepsilon$. 
\begin{question}
Is it possible to build $u_\varepsilon$ with a number of discontinuities of order 
$\displaystyle{o_{\varepsilon \to 0} \left (\frac{1}{\varepsilon}\right )}$?
\end{question}

\subsection{Good Galerkin Approximations  for general systems}

The existence of Good Galerkin Approximations is of crucial interest for the theoretical 
analysis and the numerical simulation of bilinear quantum systems. 
For systems that are not weakly-coupled (e.g., example of Section \ref{SEC_Anharmonic}), 
there is no equivalent of Proposition \ref{prop:gga} in general. However, if $(A,B)$ has the particular 
form $A=-\mathrm{i}(\Delta+V)$, $B=\mathrm{i}W$, with $\Delta$ the Laplace-Beltrami operator 
on a compact manifold $\Omega$ and $V:\Omega \to \mathbf{R}$ 
a smooth function, then for any measurable bounded $W:\Omega \to \mathbf{R}$, $(A,B)$ 
admits a Good Galerkin Approximation. This can be proved by considering 
$W_{\eta}:\Omega \to \mathbf{R}$ a smooth function $\eta$-close in $L^1$  norm to $W$. 
$(A,\mathrm{i}W_\eta)$ is $k$-weakly-coupled for every $k$, 
thus Proposition \ref{prop:gga} applies, and the trajectory of $(A,\mathrm{i}W_\eta)$ with control $u$  is $\|u\|_{L^1}\eta$ 
close to the trajectory of $(A,\mathrm{i}W)$ with control $u$. Conclusion 
follows by letting $\eta$ tend to zero.
\begin{question}
Does it exist a system $(A,B)$ with unbounded $B$ that satisfies Assumption \ref{ASS_1}, 
is not $k$-weakly-coupled for any $k>0$ and that can be approached, uniformly with respect of the 
$L^1$ norm of the control, by its Galerkin approximations?
\end{question}
Notice that the example of Section \ref{SEC_Anharmonic} with $\alpha\geq 3$ is a counter-example 
to the natural idea ``If $B$ is $A$-bounded, then $(A,B)$ admits Good Galerkin Approximations''.

\section{CONCLUSIONS}
 
The variety of approaches and methods developed by different authors in the last years to tackle
the difficult problem of the controllability of infinite dimensional bilinear quantum systems is essentially
the sign of the rich structure and subtle nature of control issues in
this context. It is likely that
new methods  will be necessary to answer the  many open problems in the fields.


\bibliographystyle{IEEEtran}
\bibliography{biblioECC}

\end{document}